\documentclass{amsart}
\usepackage{amsmath, amssymb}
\usepackage[all]{xypic}

\newtheorem{theorem}[equation]{Theorem}

\newtheorem{corollary}[equation]{Corollary}
\newtheorem{lemma}[equation]{Lemma}

\numberwithin{equation}{section}

\newcommand{\beq}{\begin{equation}}
\newcommand{\eeq}{\end{equation}}

\newcommand{\id}{{\rm Id}}

\DeclareMathOperator{\Tor}{Tor}

\newcommand{\blank}{\mbox{$\underline{\makebox[10pt]{}}$}}

\newcommand{\bbar}[1]{\overline{#1}}

\renewcommand{\P}{\mathbb{P}}

\newcommand{\struct}{\mathcal{O}}

\newcommand{\Esh}{\mathcal{E}}
\newcommand{\Fsh}{\mathcal{F}}

\renewcommand{\Lsh}{\mathcal{L}}

\newcommand{\sh}[1]{\mathcal{#1}}
\newcommand{\shTor}{\mathcal{T}\!{\it or}}
\DeclareMathOperator{\Tot}{Tot}
\DeclareMathOperator{\Spec}{Spec}
\DeclareMathOperator{\len}{len}

\newcommand{\kbar}{\bbar{k}}
\newcommand{\mf}{\mathfrak}

\title{A general homological Kleiman-Bertini theorem}
\author{Susan J. Sierra}
\address{Department of Mathematics \\
University of Washington\\
Seattle, WA 98195}
\email{ssierra@umich.edu}
\date{\today}
\thanks{The author was partially supported by NSF grants DMS-0502170 and DMS-0802935.  This paper is part of the author's Ph.D. thesis at the University of Michigan under the direction of J.T. Stafford.}

\begin{document}
\begin{abstract}
Let $G$ be a smooth algebraic group acting on a variety $X$.  Let $\sh{F}$ and $\sh{E}$ be coherent sheaves on $X$.  We show that if all the higher $\shTor$ sheaves of $\sh{F}$ against $G$-orbits vanish, then for generic $g \in G$, the sheaf $\shTor^X_j(g \sh{F}, \sh{E})$ vanishes for all $j \geq 1$.  This generalizes a result of Miller and Speyer for transitive group actions and a result of Speiser, itself generalizing the classical Kleiman-Bertini theorem, on generic transversality, under a general  group action, of smooth subvarieties over an algebraically closed field of characteristic 0.
\end{abstract}
\maketitle

\section{Introduction}\label{sec-introduction}
All schemes that we consider in this paper are of finite type over a fixed field $k$; we make no assumptions on the characteristic of $k$.  

Our  starting point  is the following result of Miller and Speyer:

\begin{theorem}\label{thm-MS}
{\em \cite{MS}}
Let $X$ be a variety with a transitive left action of a smooth algebraic group $G$.  Let $\sh{F}$ and $\sh{E}$ be coherent sheaves on $X$, and for all $k$-points $g \in G$, let $g \sh{F}$ denote the pushforward of $\sh{F}$ along multiplication by $g$.  Then there is a dense Zariski open subset $U$ of $G$ such that, for all $k$-rational points $g \in U$ and for all $j \geq 1$, the sheaf $\shTor_j^X(g \sh{F}, \sh{E})$ is zero.
\end{theorem}

As Miller and Speyer remark, their result is a homological generalization of the Kleiman-Bertini theorem: in characteristic 0, if  $\sh{F} = \struct_W$ and $\sh{E} = \struct_Y$ are structure sheaves of smooth subvarieties of $X$ and $G$ acts transitively on $X$, then $gW$ and $Y$ meet transversally for generic $g$, implying that $\struct_{gW} = g \struct_W$ and $\struct_Y$ have no higher $\shTor$.   
Motivated by this, if $\sh{F}$ and $\sh{E}$ are quasicoherent sheaves on $X$ with $\shTor_j^X(\sh{F}, \sh{E}) = 0$ for $j \geq 1$, we will say that $\sh{F}$ and $\sh{E}$ are {\em homologically transverse}; if $\sh{E} = \struct_Y$ for some closed subscheme $Y$ of  $X$, we will simply  say that $\sh{F}$ and $Y$ are homologically transverse.

Homological transversality has a geometric meaning if $\Fsh = \struct_W$ and $\Esh = \struct_Y$ are structure sheaves of closed subschemes of $X$.   If $P$ is a component of $Y \cap W$, then Serre's formula for the multiplicity of the intersection of $Y$ and $W$ at $P$ \cite[p.~427]{Ha} is:
\[ i(Y, W; P) = \sum_{j \geq 0} (-1)^j \len_P(\shTor^X_j(\Fsh, \Esh)),\]
where the length is taken over the local ring at $P$.  Thus if $Y$ and $W$ are homologically transverse, their intersection multiplicity at $P$ is simply the length of their scheme-theoretic intersection over the local ring at $P$.

It is natural to ask what conditions on the action of $G$ are necessary to conclude that homological transversality is generic in the sense of Theorem~\ref{thm-MS}.  In particular, the restriction to transitive actions is unfortunately strong, as it excludes important situations such as the torus action on $\P^n$.  On the other hand, suppose that $\Fsh$ is the structure sheaf of the closure of a non-dense orbit.  Then for all $k$-points $g \in G$, we have  $\shTor^X_1(g \sh{F}, \sh{F}) = \shTor^X_1(\Fsh, \Fsh) \neq 0$, and so the conclusion of Theorem~\ref{thm-MS} fails (as long as $G(k)$ is dense in $G$).  Thus for non-transitive group actions some additional hypothesis is necessary.

The main result of this paper is that there is a simple condition for homological transversality to be generic.  This is:

\begin{theorem}\label{thm-generalgroup}
Let $X$ be a scheme with a left action of  a smooth algebraic group $G$, and let $\Fsh$ be a coherent sheaf on $X$. Let $\kbar$ be an algebraic closure of $k$.  Consider the following conditions:
\begin{enumerate}
\item For all closed points $x \in X \times \Spec \kbar$, the pullback of $\Fsh$ to $X \times \Spec \kbar$ is homologically transverse to the closure of the $G(\kbar)$-orbit of $x$;
\item For all coherent sheaves $\Esh$ on $X$, there is a Zariski open and dense subset $U$ of $G$ such that for all $k$-rational points $g \in U$, the sheaf $g \Fsh$  is homologically transverse to $\Esh$.
\end{enumerate}
Then (1) $\Rightarrow$ (2).  If $k$ is algebraically closed,  then (1) and (2) are equivalent.
\end{theorem}

If $g$ is not $k$-rational, the sheaf $g \Fsh$ can still be defined; in Section~\ref{sec-generalizations} we give this definition and a generalization  of (2) that is equivalent to (1) in any setting (see Theorem~\ref{thm-final}).

If $G$ acts transitively on $X$ in the sense of \cite{MS}, then the action is geometrically transitive, and so (1) is trivially satisfied.  Thus  Theorem~\ref{thm-MS} follows from Theorem~\ref{thm-generalgroup}.
Since transversality of smooth subvarieties in characteristic 0 implies homological transversality,  Theorem~\ref{thm-generalgroup} also  generalizes the following result of Robert Speiser:

\begin{theorem}\label{thm-Sp}
{\em \cite[Theorem~1.3]{Sp}}
Suppose that $k$ is algebraically closed of  characteristic 0.  
Let $X$ be a smooth variety, and let $G$ be a (necessarily smooth) algebraic group acting on $X$.  Let $W$ be a smooth closed subvariety of $X$.  If $W$ is transverse to every $G$-orbit in $X$, then for any smooth closed subvariety $Y \subseteq X$, there is a dense open subset $U$ of $G$ such that if $g \in U$,  then $gW$ and $Y$ are transverse.
\end{theorem}

Speiser's result implies that the generic intersection $gW \cap Y$, for $g \in U$, is also smooth.  We also give a more general homological version of this.  For simplicity, we state it here for algebraically closed fields, although in the body of the paper (see Theorem~\ref{thm-CMG}) we remove this assumption.
 
\begin{theorem}\label{ithm-CMG}
Assume that $k = \kbar$.  Let $X$ be a scheme with a left action of a smooth algebraic group $G$, and let $W$ be a Cohen-Macaulay (respectively, Gorenstein)  closed subscheme of $X$ such that 
$W$ is homologically transverse to the $G$-orbit closure of every closed point $x \in X$.  Then for any Cohen-Macaulay (respectively, Gorenstein) closed subscheme $Y$ of $X$, there is  a dense open subset $U \subseteq G$ so that
$gW $ is homologically transverse to $Y$ and $gW \cap Y$ is Cohen-Macaulay (Gorenstein) 
 for all closed points $g \in U$.
\end{theorem}

Theorem~\ref{thm-generalgroup} was proved in the course of an investigation of certain rings, determined by geometric data, that arise in the study of noncommutative algebraic geometry.  Given  a variety $X$, an automorphism $\sigma$ of $X$ and an invertible sheaf $\Lsh$ on $X$, then Artin and Van den Bergh \cite{AV} construct a  {\em twisted homogeneous coordinate ring} $B = B(X, \sh{L}, \sigma)$.  The graded ring $B$ is  defined via $B_n = H^0(X, \Lsh \otimes_X \sigma^* \Lsh \otimes_X \cdots \otimes_X (\sigma^{n-1})^* \Lsh)$, with multiplication of sections given by the action of $\sigma$.  A closed subscheme $W$ of $X$ determines a graded right ideal $I$ of $B$, generated by sections vanishing on $W$.  In \cite{S2}, we study the {\em idealizer} of $I$; that is, the maximal subring $R$ of $B$ such that $I$ is a two-sided ideal of $R$.  It turns out that quite subtle properties of $W$ and its motion under $\sigma$ control many of the properties of $R$; in particular, for $R$ to be left Noetherian one needs that for any closed subscheme $Y$, all but finitely many $\sigma^n W$ are homologically transverse to $Y$.  (For details, we refer the reader to \cite{S2}.)  Thus we were naturally led to ask how often homological transversality can be considered ``generic'' behaviour, and what conditions on $W$ ensure this.

 We make some remarks on notation.   If  $x$ is any point of a scheme $X$, we denote the skyscraper sheaf at $x$ by $k_x$.  For schemes $X$ and $Y$, we will write $X \times Y$ for  the product $X \times_k Y$.  
Finally, if $X$ is a scheme with a (left) action of an algebraic group $G$, we will always denote the multiplication map by $\mu: G \times X \to X$.

\section{Generalizations}\label{sec-generalizations}
We begin this section by defining homological transversality  more generally.  If  $W$ and $Y$ are schemes over a scheme $X$, with (quasi)coherent sheaves $\sh{F}$ on $W$ and $\sh{E}$ on $Y$ respectively, then for all $j \geq 0$ there is a (quasi)coherent sheaf $\shTor^X_j(\sh{F}, \sh{E})$ on $W \times_X Y$.  This sheaf is defined locally.  Suppose that   $X = \Spec R$, $W = \Spec S$ and $Y = \Spec T$ are affine.  Let $(\blank)\, \widetilde{}$ denote the functor that takes an $R$-module (respectively $S$- or $T$-module) to the associated quasicoherent sheaf on $X$ (respectively $W$ or $Y$).  If $F$ is an  $S$-module and $E$ is a $T$-module, we define $\shTor^X_j(\widetilde{F}, \widetilde{E})$ to be $(\Tor^R_j(F, E))\, \widetilde{}$.  That these glue properly to give sheaves on $W \times_X Y$ for general $W$, $Y$, and $X$ is \cite[6.5.3]{G}.    As before, we will say that $\sh{F}$ and $\sh{E}$ are 
{\em  homologically transverse} if the sheaf $\shTor^X_j(\sh{F}, \sh{E})$ is zero for all $j \geq 1$. 

 We caution the reader that the maps from $W$ and $Y$ to $X$ are implicit in the definition of $\shTor_j^X(\Fsh, \Esh)$; at times we will write $\shTor_j^{W\to X \leftarrow Y}(\Fsh, \Esh)$ to make this more obvious.  
We also remark that if $Y = X$, then $\shTor^X_j(\sh{F}, \sh{E})$ is a sheaf on $W \times_X X = W$.  As localization commutes with $\Tor$, for any $w \in W$ lying over $x \in X$ we have in this case that $\shTor^X_j(\sh{F}, \sh{E})_w = \Tor^{\struct_{X, x}}_j(\Fsh_w, \Esh_x)$.

Now suppose that $f: W \to X$ is a morphism of schemes and $G$ is an algebraic group acting on $X$.   Let $\Fsh$ be a (quasi)coherent sheaf on $W$ and let $g$ be any point of $G$.  We will denote  the pullback of  $\Fsh$ to $\{ g\} \times W$ by $g \sh{F}$.    There is a map
\[ \xymatrix{
 \{ g \} \times W \ar[r] &G \times W \ar[r]^{1 \times f} & G \times X \ar[r]^{\mu} & X.}\]
If $Y$ is a scheme over $X$ and $\Esh$ is a (quasi)coherent sheaf on $Y$, we will write $ \shTor_j^X(g \Fsh, \Esh)$ for the 
 (quasi)\-coherent sheaf 
 $\shTor_j^{\{g \} \times W \to X \leftarrow Y}(g \Fsh, \Esh)$ on $W \times_X Y \times k(g)$.  
Note that if  $W = X$ and $g$ is $k$-rational, then $g \Fsh$ is simply the pushforward of $\Fsh$ along multiplication by $g$.

In this context, we prove the following relative version of Theorem~\ref{thm-generalgroup}:

\begin{theorem}\label{thm-final}
Let $X$ be a scheme with a left action of a smooth algebraic group $G$, let $f: W \to X$ be a morphism of schemes, and let $\Fsh$  be a coherent sheaf on $W$.   We define maps:
\[ \xymatrix{
G \times W \ar[r]^(.6){\rho} \ar[d]_p & X \\
W }\]
where $\rho$ is the map $\rho(g,w) = gf(w)$ induced by the action of $G$ and $p$ is projection onto the second factor.  

Then the following are equivalent:
\begin{enumerate}
\item For all closed points $x \in X \times \Spec \kbar$,  the pullback of $\Fsh$ to $W \times \Spec \kbar$ is homologically transverse to the closure of the $G(\kbar)$-orbit of $x$;
\item For all schemes $r: Y \to X$ and all coherent sheaves $\Esh$ on $Y$, there is a Zariski open and dense subset $U$ of $G$ such that for all closed points $g \in U$, the sheaf $g \Fsh$ on $\{g\} \times W$ is homologically transverse to $\Esh$.  
\item The sheaf $p^* \Fsh$ on $G \times W$ is $\rho$-flat over $X$.
\end{enumerate}
\end{theorem}
A related relative version of Theorem~\ref{thm-Sp} is given in \cite{Sp}.

Our general approach to Theorem~\ref{thm-final} mirrors that of \cite{Sp}, although the proof techniques are quite different.  We first generalize Theorem~\ref{thm-MS} to apply to any flat map  $f: W \to X$; this is a homological version of \cite[Lemma~1]{Kl} and may be of independent interest.

\begin{theorem}\label{thm-family}
Let $X$, $Y$, and $W$ be schemes, let $A$ be a generically reduced scheme, and suppose that there are morphisms:
 \[ \xymatrix{
& Y \ar[d]^{r} \\
W \ar[r]^{f}	\ar[d]_{q}	& X \\
A. } \]
  Let $\Fsh$ be a coherent sheaf on $W$ that is $f$-flat over $X$, and let $\Esh$ be a coherent sheaf on $Y$.  For all $a \in A$, let $W_a$ denote the fiber of $W$ over $a$, and let $\Fsh_a = \Fsh \otimes_W \struct_{W_a}$ be the fiber of $\Fsh$ over $a$.   
    
Then  there is a dense open $U \subseteq A$ such that if $a \in U$, then $\Fsh_a$ is homologically transverse to $\Esh$.  
\end{theorem}

We note that we have not assumed  that $X$, $Y$, $W$, or $A$ is smooth.

\section{Proofs}\label{sec-proofs}
In this section we prove Theorem~\ref{thm-generalgroup}, Theorem~\ref{thm-final}, and Theorem~\ref{thm-family}.
We begin by establishing some preparatory lemmas.

\begin{lemma}\label{lem1}
Let 
\[ \xymatrix{ X_1 \ar[r]^{\alpha} & X_2 \ar[r]^{\gamma} & X_3}\]
be morphisms of schemes, and assume that $\gamma$ is flat. Let $\sh{G}$ be a quasicoherent sheaf on $X_1$ that is flat over $X_3$.  Let $\sh{H}$ be any quasicoherent sheaf on $X_3$.  Then for all $j \geq 1$, we have $\shTor^{X_2}_j(\sh{G}, \gamma^* \sh{H}) = 0$.
\end{lemma}

\begin{proof} We may reduce to the local case.  Thus let $x \in X_1$ and let $y = \alpha(x)$ and $z = \gamma(y)$. Let $S = \struct_{X_2,y}$ and let $R = \struct_{X_3, z}$.  Then  $(\gamma^* \sh{H})_y \cong S \otimes_R \sh{H}_{z}$.  Since  $S$ is flat over $R$,   we have
\[\Tor_j^{R}(\sh{G}_x, \sh{H}_{z}) 
 \cong 
\Tor_j^{S}(\sh{G}_x, S \otimes_R \sh{H}_{z}) 
= \shTor_j^{X_2}(\sh{G}, \gamma^* \sh{H})_{x}\]
by flat base change.
The left-hand side is 0 for $j \geq 1$ since $ \sh{G}$ is flat over $X_3$.  
Thus for $j \geq 1$ we have 
$\shTor_j^{X_2}(\sh{G}, \gamma^* \sh{H}) = 0$.  
\end{proof}

To prove Theorem~\ref{thm-family}, we show that a suitable modification of the spectral sequences used in \cite{MS} will work in our situation.  Our key computation is the following lemma; compare to \cite[Proposition~2]{MS}.

\begin{lemma} \label{lem3}

Given the notation of Theorem~\ref{thm-family}, there is an open dense $U \subseteq A$ such that for all $a \in U$ and for all $j \geq 0$ we have 
\[ \shTor^W_j( \Fsh \otimes_X  \Esh, q^*k_a) \cong \shTor_j^X( \Fsh_a, \Esh)\]
as sheaves on $W \times_X Y$.
\end{lemma}

Note that $\Fsh \otimes_X \Esh$ is a sheaf on $W \times_X Y$ and thus $\shTor^W_j(\Fsh \otimes_X \Esh, q^* k_a)$ is a sheaf on $W \times_X Y \times_W W = W \times_X Y$ as required.

\begin{proof}
Since $A$ is generically reduced, we may apply generic flatness to the morphism $q: W\to A$.  Thus there is an open dense subset $U$ of $A$ such that  both $W$ and $\Fsh$ are flat over $U$.   Let $a \in U$.   Away from $q^{-1} (U)$, both sides of the equality we seek to establish are zero, and so the result is trivial.  Since $\sh{F}|_{q^{-1}(U)}$ is still flat over $X$,  without loss of generality we may replace $W$ by $q^{-1}(U)$; that is, we may assume that both $W$ and $\Fsh$ are flat over $A$.

The question is local, so assume that $X = \Spec R$, $Y = \Spec T$, and $W = \Spec S$ are affine.  Let $E = \Gamma(Y, \sh{E})$ and let $F = \Gamma(W, \sh{F})$.  Let $Q = \Gamma(W, q^*k_a)$; then $\Gamma(W, \Fsh_a) = F \otimes_S Q$.  We seek to show that 
\[ \Tor^S_j(F \otimes_R E, Q) \cong \Tor^R_j(F \otimes_S Q, E)\]
as $S \otimes_R T$-modules.

We will work on $W \times X$.  For clarity, we lay out the various morphisms and corresponding ring maps in our situation.   We have morphisms of schemes
\[ \xymatrix{
W \times X \ar[d]_p	& Y \ar[d]^r \\
W \ar@/_1pc/[u]_{\phi} \ar[r]_f		& X }\]
where $p$ is projection onto the first factor and the morphism ${\phi}$ splitting $p$ is given by the graph of $f$.  Letting $B = S \otimes_k R$, we have corresponding  maps of rings
\[ \xymatrix{
B \ar@/^1pc/[d]^{{\phi}^{\#}}	& T \\
S \ar[u]^{p^{\#}}	& R, \ar[u]_{r^{\#}} \ar[l]^{f^{\#}} }\]
where $p^{\#}(s) = s \otimes 1$ and ${\phi}^{\#}(s \otimes r) = s \cdot f^{\#} (r)$.   
We make the trivial observation that
\[ B \otimes_R E = (S \otimes_k R) \otimes_R E \cong S \otimes_k E.\]

 Let $K_{\bullet} \to F$ be a projective resolution of $F$, considered as a $B$-module via the map $\phi^{\#}: B \to S$.   As $E$ is an $R$-module via the map $r^{\#}:R \to T$, there is a $B$-action on $S \otimes_k E$; let $L_{\bullet} \to S \otimes_k E$ be a projective resolution over $B$.  

Let $P_{\bullet, \bullet}$ be the double complex $K_{\bullet} \otimes_B L_{\bullet}$.   We claim the  total complex of $P_{\bullet, \bullet}$ resolves $F \otimes_B (S \otimes_k E)$.   To see this, note that the rows of $P_{\bullet, \bullet}$, which are of the form $K_{\bullet} \otimes_B L_j$, are acyclic, except in degree 0, where the homology is $F \otimes_B L_j$.  The degree 0 horizontal homology forms a vertical complex whose homology computes $\Tor^B_j(F, S \otimes_k E)$.  But $S \otimes_k E \cong B \otimes_R E$, and $B$ is a flat $R$-module.  Therefore 
$\Tor^B_j(F, S\otimes_k E) \cong \Tor^B_j(F, B \otimes_R E) \cong \Tor^R_j(F, E)$ by the formula for flat base change for Tor.  Since $F$ is flat over $R$, this is zero for all $j \geq 1$.   Thus, via the spectral sequence
\[ H^v_j (H^h_i P_{\bullet, \bullet}) \Rightarrow H_{i+j} \Tot P_{\bullet, \bullet} \]
we see that the total complex of $P_{\bullet, \bullet}$ is acyclic, except in degree 0, where the homology is $F \otimes_B S \otimes_k E \cong F \otimes_R E$.

Consider the double complex $P_{\bullet, \bullet} \otimes_S Q$.  Since $\Tot P_{\bullet, \bullet}$ is a $B$-projective and therefore $S$-projective resolution of $F \otimes_R E$, the homology of the total complex of this double complex computes $\Tor^S_j(F \otimes_R E, Q)$.  

Now consider the row $K_\bullet \otimes_B L_j \otimes_S Q$.  As $L_j$ is $B$-projective and therefore $B$-flat, the $i$'th homology of this row is isomorphic to  $\Tor^S_i(F, Q) \otimes_B L_j$. Since $W$ and $\Fsh$ are flat over $A$, by Lemma~\ref{lem1} we have $\Tor^S_i(F, Q) = 0$ for all $i \geq 1$.  Thus this row is acyclic except in degree 0, where the homology is $F \otimes_B L_j \otimes_S Q$.
The vertical differentials on the degree 0 homology give a complex whose $j$'th homology is isomorphic to  $\Tor^B_j(F \otimes_S Q, S \otimes_k E)$.  As before, this is simply $\Tor^R_j(F \otimes_S Q, E)$.

Thus (via a spectral sequence) we see that the homology of the total complex of $P_{\bullet, \bullet} \otimes_S Q$ computes $\Tor^R_j(F \otimes_S Q, E)$.  But we have already seen that the homology of this total complex is isomorphic to $\Tor^S_j(F \otimes_R E, Q)$.  Thus the two are isomorphic.  
\end{proof}

\begin{proof}[Proof of Theorem~\ref{thm-family}]
By generic flatness, we may reduce without loss of generality to the case where $W$ is flat over $A$.  
Since $\Fsh$ and $\Esh$ are coherent sheaves on $W$ and $Y$ respectively, $\Fsh \otimes_X \Esh$ is a coherent sheaf on $W \times_X Y$.  Applying generic flatness to the composition $W \times_X Y \to W \to A$, we obtain a dense open $V \subseteq A$ such that $\sh{F} \otimes_X  \Esh$ is flat over $V$.   Therefore, by Lemma~\ref{lem1}, if $a \in V$ and $j \geq 1$, we have $\shTor^W_j(\sh{F} \otimes_X \Esh, q^*k_a) = 0$.

We apply Lemma~\ref{lem3} to choose a dense open $U \subseteq A$ such that  for all $j \geq 1$, if 
$a \in U$, then $\shTor^W_j(\Fsh \otimes_X \Esh, q^*k_a) \cong  \shTor_j^X(\Fsh_a,\Esh)$.  Thus if 
$a$ is in the  dense open set $U \cap V$, then  for all $j \geq 1$ we have 
\[  \shTor_j^X(\Fsh_a, \Esh) \cong \shTor^W_j(\sh{F} \otimes_X  \Esh, q^*k_a) = 0,\]
as required.
\end{proof}

We now turn to the proof of Theorem~\ref{thm-final}; for the remainder of this paper, we will adopt the hypotheses and notation given there.  

\begin{lemma}\label{lemff}
Let $R, R', S$, and $T$ be commutative rings, and let 
\[ \xymatrix{
R' \ar[r]	& T \\
R \ar[u] \ar[r]& S \ar[u]   }\]
be a commutative diagram of ring homomorphisms, such that $R'_R$  and $T_S$ are flat.  Let $N$ be an $R$-module.  Then for all $j \geq 0$, we have that
\[\Tor_j^{R'}(N \otimes_R R', T) \cong \Tor^R_j(N, S) \otimes_S T.\]
\end{lemma}
\begin{proof}
Let $P_{\bullet} \to N$ be a projective resolution of $N$.   Consider  the complex
\beq \label{starcx}
P_{\bullet} \otimes_R R' \otimes_{R'} T \cong  P_{\bullet} \otimes_R T \cong P_{\bullet} \otimes_R S \otimes_S T.
 \eeq
Since $R'_R$ is flat,  $P_{\bullet} \otimes_R R'$ is a projective resolution of $N \otimes_R R'$.  Thus the j'th homology of \eqref{starcx} computes $\Tor_j^{R'} (N \otimes_R R', T)$.  Since $T_S$ is flat, this homology is isomorphic to 
$H_j(P_{\bullet} \otimes_R S) \otimes_S T$.  Thus  $\Tor_j^{R'}(N \otimes_R R', T) \cong \Tor^R_j(N, S) \otimes_S T$.   
\end{proof}

\begin{lemma}\label{lem-comp}
Let $x$ be a closed point of $X$.   Consider the multiplication map 
\[\mu_x: G \times \{ x\} \to X.\] 
Then for all $j \geq 0$ we have
\beq\label{firsteq}
\shTor^X_j(\Fsh, \struct_{G \times \{x\}}) \cong \shTor^{G \times X}_j(p^* \Fsh, \mu^* k_x)
\eeq
If $k$ is algebraically closed, then we also have 
\beq\label{secondeq}
 \shTor^{G \times X}_j(p^* \Fsh, \mu^* k_x) \cong \shTor^X_j(\Fsh, \struct_{\bbar{Gx}}) \otimes_X \struct_{G\times\{x\}}.
\eeq
 All isomorphisms are of sheaves on $G \times W$.
\end{lemma}
\begin{proof}
Note that $\mu_x$ maps $G \times \{x\}$ onto a locally closed subscheme of $X$, which we will denote $Gx$.   Since all computations may be done locally,  without loss of generality we may assume that $Gx$ is in fact a closed subscheme of $X$.  

 Let $\nu:  G \to G$ be the inverse map, and let $\psi = \nu \times \mu: G \times X \to G\times X$.  
Consider the  commutative diagram:
\beq\label{diag1} \xymatrix{
G \times W \ar[d]_p \ar[r]^{1 \times f} & G\times X \ar[d]_p	& G \times \{x\} \ar[l]_{\psi} \ar[d]^{\pi} \ar[ld]_{\mu_x} \\
W \ar[r]_{f} & X & Gx \ar[l]_{\supseteq} }\eeq
where  $\pi$ is the induced map and $p$ is projection onto the second factor.    
Since $\psi^2  = \id_{G \times X}$ and $\mu = p \circ \psi$, we have that
$\mu^* k_x \cong \psi^* p^* k_x \cong \psi_* \struct_{G \times \{x\}}$, considered as sheaves on $G \times X$.  
Then the isomorphism \eqref{firsteq} is a direct consequence of the flatness of $p$ and Lemma~\ref{lemff}.  If $k$ is algebraically closed, then $\pi$ is also flat, and so the isomorphism 
 \eqref{secondeq} 
also follows from Lemma~\ref{lemff}.
 \end{proof}

\begin{proof}[Proof of Theorem~\ref{thm-final}]
(3) $\Rightarrow$ (2).  Assume (3).    
Let $\Esh$ be a coherent sheaf on $Y$.  Consider the maps:
\[ \xymatrix{
& Y \ar[d]^r \\
G \times W \ar[r]^(.6){\rho} \ar[d]_q	& X \\
G,}\]
where $q$ is projection on the first factor.  

Since $G$ is smooth, it is generically reduced.   
 Thus we may apply Theorem~\ref{thm-family} to the $\rho$-flat sheaf $p^* \sh{F}$ to obtain  a dense open $U \subseteq G$ such that if $g \in U$ is a closed point, then $\rho$ makes $(p^* \sh{F})_g$ homologically transverse to $\Esh$.   But $\rho |_{\{g\} \times W}$ is the map used to define $\shTor_j^X(g \Fsh, \Esh)$;
 that is, considered as sheaves over $X$, $(p^*\Fsh)_g \cong g\Fsh$.  Thus (2) holds.

(2) $\Rightarrow$ (3).  The morphism $\rho$ factors as
\[ \xymatrix{ G \times W \ar[r]^{1 \times f} & G \times X \ar[r]^{\mu} & X.}\]
  Since the multiplication map $\mu$ is  the composition of an automorphism  of $G \times X$  and projection, it is flat.    
Therefore for any quasicoherent $\sh{N}$ on $X$ and $\sh{M}$ on $G \times W$ and for any closed point $z \in G \times W$, we have 
\beq\label{dag}
\shTor^{G \times X}_j(\sh{M}, \mu^* \sh{N})_{z} \cong \shTor_j^{\struct_{X, \rho(z)}}(\sh{M}_{z}, \sh{N}_{\rho(z)}),
\eeq
 as in the proof of Lemma~\ref{lem1}.  
 
If $p^* \Fsh$ fails to be flat over $X$, then  flatness fails against the structure sheaf of some closed point $x \in X$, by the local criterion for flatness \cite[Theorem~6.8]{Eis}.   Thus to check that  $p^* \Fsh$ is flat over $X$, it is equivalent to test flatness  against structure sheaves of closed points of $X$.   
By \eqref{dag}, we see that $p^* \sh{F}$ is $\rho$-flat over $X$ if and only if 
\beq \label{blah}
\shTor^{G \times X}_j(p^* \Fsh, \mu^* k_x) = 0  \quad \mbox{ for all closed points $x \in X$ and for all $j \geq 1$. }
\eeq 
Applying Lemma~\ref{lem-comp}, we see that the flatness of $p^* \Fsh$ is  equivalent to the vanishing
\beq \label{biff}
\shTor^{X}_j( \Fsh, \struct_{G \times \{x\}}) = 0  \quad \mbox{ for all closed points $x \in X$ and for all $j \geq 1$. }
\eeq 

Assume (2).  We will show that \eqref{biff} holds for all $x \in X$.  Fix a closed point $x \in X$ and consider the morphism $ \mu_x: G \times \{x \}\to X$. 
By assumption, there is a closed point $g \in G$ such that $g \Fsh$ is homologically transverse to $\struct_{G \times \{x\}}$.   Let $k' = k(g)$ and let $g'$ be the canonical $k'$-point of $G \times \Spec k'$ lying over $g$. Let $G' = G \times \Spec k'$ and let $X' = X \times  \Spec k'$.  Let $\Fsh'$ be the pullback of $\Fsh$ to $W' = W \times \Spec k'$.  Consider the commutative diagram
 \[ \xymatrix{
G \times \{ x\} \times \Spec k'  \ar[rr]^(.6) {\mu_x \times 1} \ar[d] 	&&  X' \ar[d]	&  \{ g' \} \times_{k'} W' \ar[l]_(.65){\rho} \ar[d]_{\cong}\\
 G \times \{x\} \ar[rr]^(.6){\mu_x}	&& X &  \{ g \} \times W. \ar[l]_(.6){\rho}  }\]

  Since the vertical maps are faithfully flat and the left-hand square is a fiber square, by Lemma~\ref{lemff}  $g' \Fsh'$ is homologically transverse to 
  \[G \times \{x\} \times \Spec k' \cong G' \times \{x\}.\]
    By $G(k')$-equivariance,  $\Fsh'$ is homologically transverse to $(g')^{-1} G' \times \{x\} = G' \times \{x\}$.  Since 
  \[ \xymatrix{G' \times \{x\} \ar[r] & X' & W'  \ar[l]_f }\]
  is base-extended from 
  \[ \xymatrix{G \times \{x\} \ar[r]& X & W, \ar[l]_f }\]
  we obtain that $\Fsh$ is homologically transverse to $G \times \{x\}$.  Thus \eqref{biff} holds.
  
  (1) $\Rightarrow$ (3).   
  The $\rho$-flatness of $\Fsh$ is not affected by base extension, so without loss of generality we may assume that $k$ is algebraically closed.  Then (3) follows directly from Lemma~\ref{lem-comp} and the criterion \eqref{blah} for flatness.  
  
  (3) $\Rightarrow$ (1).  As before, we may assume that $k$ is algebraically closed.  Let $x$ be a closed point of $X$.  We have seen that (3) and (2) are equivalent; by (2) applied to $\Esh = \struct_{\bbar{Gx}}$ there is a closed point $g \in G$ such that $g \Fsh$ and $\bbar{Gx}$ are homologically transverse.  By $G(k)$-equivariance, $\Fsh$ and $g^{-1} \bbar{Gx} = \bbar{Gx}$ are homologically transverse.  
 \end{proof}

\begin{proof}[Proof of Theorem~\ref{thm-generalgroup}]
If $\Fsh$ is homologically transverse to orbit closures upon extending to $\kbar$, then, using Theorem~\ref{thm-final}(2), for any $\Esh$ there is a dense open $U \subseteq G$ such that, in particular, for any $k$-rational $g \in U$ the sheaves $g \Fsh$ and $\Esh$ are homologically transverse.

The equivalence of (1) and (2) in the case that $k$ is algebraically closed follows directly from Theorem~\ref{thm-final}.
\end{proof}

We recall that Theorem~\ref{thm-generalgroup} is a statement about $k$-rational points in $U \subseteq G$.  However, the proof shows that for any extension $k'$ of $k$ and any $k'$-rational $g \in U \times \Spec k'$, then $g \Fsh$ will be  homologically transverse to $\Esh$  on $X \times \Spec k'$.     Further, in many situations $U$ will automatically contain a $k$-rational point of $G$.  This holds, in particular, if $k$ is infinite, $G$ is connected and affine, and either  $k$ is perfect or $G$ is reductive, by \cite[Corollary~18.3]{B}.

\section{Singularities of generic intersections}
We now specialize to consider generic intersections of two subschemes of $X$.  That is, let $X$ be a scheme with a left action of a smooth algebraic group $G$.  Let $Y$ and $W$ be closed subschemes of $X$.  
By Theorem~\ref{thm-Sp}, if $k$ is algebraically closed of characteristic 0, $W$ is transverse to $G$-orbit closures, and  $X$, $Y$, and $W$ are smooth, then for generic $g \in G$ the subschemes $gW$ and $Y$ meet transversally, and so by definition $gW \cap Y$ is smooth.  Here we remark that a homological version of this  result holds more generally:  if  $W$ is homologically transverse to $G$-orbit closures and $Y$ and $W$ are  Cohen-Macaulay (respectively, Gorenstein), then their generic intersection will  also be Cohen-Macaulay (Gorenstein).  We use the following result from commutative algebra:
\begin{theorem}\label{thm-Mats}
Let $A \to B$ be a local homomorphism of Noetherian local rings, and let $\mf{m}$ be the maximal ideal of $A$ and $F = B/\mf{m}B$.  Assume that $B$ is flat over $A$.  Then $B$ is Cohen-Macaulay (respectively, Gorenstein) if and only if $B$ and $F$ are both Cohen-Macaulay (respectively, Gorenstein).
\end{theorem}
\begin{proof}
See 
\cite[Corollary~23.3, Theorem~23.4]{M}.
\end{proof}

\begin{theorem}\label{thm-CMG}
Let $X$ be a scheme with a left action of a smooth algebraic group $G$.  Let $f: W \to X$ and $r: Y \to X$ be morphisms of schemes, such  that $W \times \Spec \bbar{k}$ is homologically transverse to the $G(\bbar{k})$-orbit of $x$  for all closed points $x \in X \times \Spec \bbar{k}$.   Further suppose that $Y$ and $W$ are Cohen-Macaulay (respectively, Gorenstein).  
Then there is a dense open subset $U \subseteq G$ so that for all closed points $g \in U$, the scheme
$\{ g\} \times W$ is homologically transverse to $Y$ and the fiber product 
$(\{ g \} \times W) \times_X Y$ is Cohen-Macaulay (respectively, Gorenstein).  
\end{theorem}
\begin{proof}
Let $\rho:  G \times W \to X$ be the map $\rho(g,w) = g f(w) $ induced by $f$ and the action of $G$. Let $q:  G \times W \to G$ be projection to the first factor.   Thus there is a commutative diagram
\[ \xymatrix{
G \times W \times_X Y \ar[r]^(.7){\rho \times 1} \ar[d]_{1 \times r}	& Y \ar[d]^{r} \\
G \times W \ar[r]^{\rho} \ar[d]_{q}	& X \\
G.	}\]
By Theorem~\ref{thm-final} applied to $\sh{F} = \struct_W$, $\rho$ is flat.  Now, $G \times W$ is Cohen-Macaulay (respectively, Gorenstein), and so by Theorem~\ref{thm-Mats}, the fibers of $\rho$ are also Cohen-Macaulay (Gorenstein).  As $Y$ is Cohen-Macaulay (Gorenstein) and $\rho \times 1$ is flat, applying Theorem~\ref{thm-Mats} again, we see that $G \times W \times_X Y$ is also Cohen-Macaulay (Gorenstein).  Now, by generic flatness and Theorem~\ref{thm-final}, there is a dense open $U \subset G$ such that $q \circ (1 \times r)$ is flat over $U$ and   $\{ g\}  \times W$ is homologically transverse to $Y$ for all $g \in U$.  For $g \in U$, the fiber $(\{ g\} \times W )\times_X Y$ of $q \circ (1 \times r)$ is Cohen-Macaulay (Gorenstein), by Theorem~\ref{thm-Mats} again.
\end{proof}

We note that, although we did not assume that $X$ is Cohen-Macaulay (respectively, Gorenstein), it follows from the flatness of $\rho$ and from Theorem~\ref{thm-Mats}.  

We also remark that  if $Y$ and $W$ are homologically transverse local complete intersections in a smooth $X$, it is not hard to show directly that $Y \cap W$ is also a local complete intersection.  We do not know if it is true in general that the homologically transverse intersection of two Cohen-Macaulay  subschemes is Cohen-Macaulay, although it follows, for example, from 
\cite[Lemma, p. 108]{F} if $X$ is smooth.

Theorem~\ref{ithm-CMG} follows directly from Theorem~\ref{thm-CMG}.

Thus we may refine Theorem~\ref{thm-MS} to obtain a result on transitive group actions that echoes the Kleiman-Bertini theorem even more closely.

\begin{corollary}\label{cor-CMG}
Let $X$ be a scheme with a geometrically transitive left action of a smooth algebraic group $G$.  Let $Y$ and $W$ be   Cohen-Macaulay (respectively, Gorenstein) closed subschemes of $X$. Then there is a dense Zariski open subset $U$ of $G$ such that $gW$ is homologically transverse to $Y$  and $gW  \cap Y$ is Cohen-Macaulay (respectively, Gorenstein) for all $k$-rational points $g \in U$. \qed
\end{corollary}

{\bf Acknowledgements.}  The author is grateful to Ezra Miller for his extraordinarily careful reading of an earlier version of this paper and for several corrections and discussions, to David Speyer for many informative conversations, and to  Mel Hochster, Kyle Hofmann, Gopal Prasad, and Karen Smith for their suggestions and assistance with 
references.  The author particularly wishes to thank Brian Conrad for finding an error in an earlier version of this paper and for several helpful discussions.    The author thanks  Susan Colley and Gary Kennedy for calling her attention to \cite{Sp}.

\end{document}